\documentclass[11pt,letterpaper]{amsart}

\usepackage{amssymb}
\usepackage{enumerate}
\usepackage{cite}

\newtheorem{theorem}{Theorem}[section]
\newtheorem{proposition}[theorem]{Proposition}
\newtheorem{corollary}[theorem]{Corollary}
\newtheorem{lemma}[theorem]{Lemma}
\newtheorem{conjecture}[theorem]{Conjecture}
\theoremstyle{remark}

\newcommand{\FF}{\mathbb{F}}

\DeclareMathOperator{\Gal}{Gal}

\begin{document}

\title[Low-degree planar polynomials]{Low-degree planar polynomials over finite fields of characteristic two}

\author{Daniele Bartoli}
\address{Department of Mathematics and Computer Science, University of Perugia, Perugia, 06123, Italy.}

\email{daniele.bartoli@unipg.it}

\author{Kai-Uwe Schmidt}
\address{Department of Mathematics, Paderborn University, Warburger Str.\ 100, 33098 Paderborn, Germany.}

\email{kus@math.upb.de}

\date{17 September 2018}

\begin{abstract}
Planar functions are mappings from a finite field $\FF_q$ to itself with an extremal differential property. Such functions give rise to finite projective planes and other combinatorial objects. There is a subtle difference between the definitions of these functions depending on the parity of $q$ and we consider the case that $q$ is even. We classify polynomials of degree at most $q^{1/4}$ that induce planar functions on $\FF_q$, by showing that such polynomials are precisely those in which the degree of every monomial is a power of two. As a corollary we obtain a complete classification of exceptional planar polynomials, namely polynomials over $\FF_q$ that induce planar functions on infinitely many extensions of~$\FF_q$. The proof strategy is to study the number of $\FF_q$-rational points of an algebraic curve attached to a putative planar function.~Our methods also give a simple proof of a new partial result for the classification of almost perfect nonlinear~functions.
\end{abstract}

\maketitle

\enlargethispage{2.1ex}

\thispagestyle{empty}


\section{Introduction and Results}

Let $q$ be a prime power. If $q$ is odd, a function $f:\FF_q\to\FF_q$ is \emph{planar} or \emph{perfect nonlinear} if, for each nonzero $\epsilon\in\FF_q$, the function
\begin{equation}
x\mapsto f(x+\epsilon)-f(x)   \label{eqn:def_planar_odd}
\end{equation}
is a permutation on $\FF_q$. Such planar functions can be used to construct finite projective planes~\cite{DemOst1968}, relative difference sets~\cite{GanSpe1975}, error-correcting codes~\cite{CarDinYua2005}, and S-boxes in block ciphers~\cite{NybKnu1993}.
\par
If $q$ is even, a function $f:\FF_q\to\FF_q$ cannot satisfy the above definition of planar functions because $x=a$ and $x=a+\epsilon$ are mapped by~\eqref{eqn:def_planar_odd} to the same image. This is the motivation to define a function $f:\FF_q\to\FF_q$ for even $q$ to be \emph{almost perfect nonlinear} (APN) if~\eqref{eqn:def_planar_odd} is a $2$-to-$1$ map. Such functions are highly relevant again for the construction of S-boxes in block ciphers~\cite{NybKnu1993}. However, there is no apparent link between APN functions and projective planes. More recently, Zhou~\cite{Zho2013} defined a natural analogue of planar functions on finite fields of characteristic two: If $q$ is even, a function $f:\FF_q\to\FF_q$ is \emph{planar} if, for each nonzero $\epsilon\in\FF_q$, the function
\[
x\mapsto f(x+\epsilon)+f(x)+\epsilon x
\]
is a permutation on $\FF_q$. As shown by Zhou~\cite{Zho2013} and Schmidt and Zhou~\cite{SchZho2014}, such planar functions have similar properties and applications as their counterparts in odd characteristic.
\par
We refer to~\cite{Pot2016} for an excellent survey of recent results for the functions defined above.
\par
The main result of this paper is a classification of the latter type of planar functions, namely those defined in characteristic two. Recall that every function from $\FF_q$ to itself is induced by a polynomial in $\FF_q[X]$ of degree at most $q-1$. A polynomial $f\in\FF_q[X]$ is called a \emph{$2$-polynomial} if the degree of every monomial in $f$ is a power of two. For even $q$, such polynomials trivially induce planar functions on $\FF_{q^r}$ for all $r\ge 1$. We show that among all polynomials of sufficiently small degree there are no other planar functions in characteristic two.
\par
From now on $q$ will always be a power of two.
\begin{theorem}
\label{thm:planar}
Let $f\in\FF_q[X]$ be a polynomial of degree at most $q^{1/4}$. If $f$ is planar on $\FF_q$, then $f$ is a $2$-polynomial. 
\end{theorem}
\par
Now consider polynomials $f\in\FF_q[X]$ with the property that $f$ is planar on $\FF_{q^r}$ for infinitely many positive integers $r$. As in~\cite{CauSchZho2016}, we call such a polynomial an \emph{exceptional} planar polynomial. As a corollary, we obtain a complete classification of such polynomials.
\begin{corollary}
\label{cor:planar}
If $f\in\FF_q[X]$ is an exceptional planar polynomial, then~$f$ is a $2$-polynomial.
\end{corollary}
\par
Theorem~\ref{thm:planar} considerably strengthens the main result of~\cite{MulZie2015}, which is the specialisation of Theorem~\ref{thm:planar} to the case that $f$ is a monomial. It should be noted that there are examples of planar functions on $\FF_q$ for even $q$ that are not induced by $2$-polynomials, see~\cite{Zho2013,SchZho2014,SchZie2014,HuLiZhaFenGe2015,Qu2016}. Of course all of these examples have degree larger than $q^{1/4}$.
\par
Our methods also give a simple proof of a partial classification result for APN functions. As in~\cite{AubMcGRod2010}, we call a polynomial $f\in\FF_q[X]$ an \emph{exceptional} APN polynomial if $f$ induces an APN function on $\FF_{q^r}$ for infinitely many positive integers $r$. It is well known (see~\cite{Pot2016}, for example) that, for each positive integer~$k$, the monomials $X^{2^k+1}$ and $X^{4^k-2^k+1}$ are exceptional APN polynomials, also called \emph{Gold} and \emph{Kasami-Welch} monomials, respectively. In fact these monomials induce APN functions on $\FF_{2^r}$ for all positive integers~$r$ that are coprime to~$k$.
\par
The following conjecture was proposed by Aubry, McGuire, and Rodier~\cite{AubMcGRod2010}.
\begin{conjecture}[\!\!\cite{AubMcGRod2010}]
\label{con:APN}
If $f\in\FF_q[X]$ is an exceptional APN polynomial, then~$f$ is equivalent to a Gold or a Kasami-Welch monomial.
\end{conjecture}
\par
In this conjecture, equivalence refers to CCZ-equivalence, whose precise definition is not required here (see~\cite{BudCarPot2006} for details). Conjecture~\ref{con:APN} has been proved by Hernando and McGuire~\cite{HerMcG2011} in the case that $f$ is a monomial and many other special cases have been proved in several papers. We refer to~\cite{Del2017} for a nice survey of the extensive recent literature on Conjecture~\ref{con:APN}. 
\par
We give a simple proof of the following new result.
\begin{proposition}
\label{pro:APN}
Let $f\in\FF_q[X]$ be a polynomial of even degree $d$ at most~$q^{1/4}$. If $f$ is APN on $\FF_q$, then $4\mid d$.
\end{proposition}
\par
Proposition~\ref{pro:APN} solves one of the five pending cases listed in~\cite[Section~4]{Del2017} and strengthens~\cite[Theorem 2.4]{AubMcGRod2010} essentially by removing the additional assumption that~$f$ has a term of odd degree.
\par
In our proof of Theorem~\ref{thm:planar} we study an algebraic surface depending on a polynomial $f\in\FF_q[X]$, such that if $f\in\FF_q[X]$ induces a planar function on~$\FF_q$, then the surface has only very few $\FF_q$-rational points. This surface is then intersected with a plane and we consider the resulting algebraic curve. The difficult part is to show that this curve has a component defined by an absolutely irreducible polynomial with coefficients in $\FF_q$. The Hasse-Weil bound then asserts that the curve, and so also the surface, has many $\FF_q$-rational points, provided that the degree of $f$ is not too large. This leads to a contradiction unless $f$ is a $2$-polynomial.
\par
This approach seems to be first used by Janwa and Wilson~\cite{JanWil1993} for monomial APN functions and by Aubry, McGuire, and Rodier~\cite{AubMcGRod2010} for general APN functions. Besides the classification problem for APN functions, classification problems for other combinatorial objects have been attacked with this method, for example for planar functions in odd characteristic~\cite{Led2015,Zie2015,CauSchZho2016}, hyperovals~ \cite{HerMcG2012,Zie2015,CauSch2015}, and maximum scattered linear sets~\cite{BarZho2018}. However a complete classification, as in Corollary~\ref{cor:planar}, has been obtained so far only in one other case, namely in the classification problem for polynomials that induce hyperovals in finite Desarguesian planes~\cite{CauSch2015}. We also remark that our methods for proving absolute irreducibility differ considerably from previous techniques.


\section{Proof strategy}

In this section we present the principal approach for proving Theorem~\ref{thm:planar}. In Section~\ref{sec:APN} we then describe the required modifications of this approach to obtain a simple proof of Proposition~\ref{pro:APN}.
\par
Let $q$ be a power of two and let $f\in\FF_q[X]$ be a nonzero polynomial in which the degree of every monomial is not a power of two. Note that there is no loss of generality here since the addition of a $2$-polynomial preserves the planarity of the function induced by $f$. Define the polynomial
\[
\phi(X,Y,W)=\frac{f(X+W)+f(X)+WX+f(Y+W)+f(Y)+WY}{(X+Y)W}.
\]
It is a direct consequence of the definition of planar functions that $f$ induces a planar function on $\FF_q$ if and only if all $\FF_q$-rational points on the affine surface defined by $\phi(X,Y,W)=0$ satisfy $X=Y$ or $W=0$.
\par
Put $\psi(X,Y,Z)=\phi(X,Y,X+Z)$, so that
\[
\psi(X,Y,Z)=1+\frac{f(X)+f(Y)+f(Z)+f(X+Y+Z)}{(X+Y)(X+Z)}.
\]
Then $f$ induces a planar function on $\FF_q$ if and only if all $\FF_q$-rational points of the affine surface defined by $\psi(X,Y,Z)=0$ satisfy $X=Y$ or $X=Z$. Now write
\[
f=\sum_{i=0}^dA_iX^i,
\]
where $A_d\ne 0$. Since $d$ is not a power of two, the homogenised form of $\psi$ is
\[
\widetilde{\psi}(X,Y,Z,T)=T^{d-2}+\sum_{i=3}^dA_i\frac{X^i+Y^i+Z^i+(X+Y+Z)^i}{(X+Y)(X+Z)}T^{d-i}.
\]
We study the intersection of the projective surface defined by $\widetilde{\psi}(X,Y,Z,T)=0$ with the plane defined by $Z=X+1$. In fact, we consider the affine curve defined by $F(X,Y)=0$, where $F(X,Y)=\widetilde{\psi}(X,1,X+1,Y)$. We have
\[
F(X,Y)=Y^{d-2}+\sum_{i=3}^dA_i\frac{X^i+1+(X+1)^i}{X+1}Y^{d-i}
\]
and, after expanding,
\begin{equation}
F(X,Y)=Y^{d-2}+\sum_{i=3}^dA_iY^{d-i}\sum_{k=0}^{i-1}\Bigg[\binom{i-1}{k}+1\Bigg]X^k.   \label{eqn:def_F}
\end{equation}
If $f$ induces a planar function on $\FF_q$, then all $\FF_q$-rational points of the affine curve defined by $F(X,Y)=0$ satisfy $X=1$ or $Y=0$.
\par
The following result is a consequence of the Hasse-Weil bound for the number of $\FF_q$-rational points on curves.
\begin{proposition}
\label{pro:Hasse_Weil}
Let $f\in\FF_q[X]$ be a polynomial of degree at most $q^{1/4}$ in which the degree of every monomial is not a power of two. If $F$ has an absolutely irreducible factor over $\FF_q$, then $f$ does not induce a planar function on $\FF_q$.
\end{proposition}
\begin{proof}
Let $d$ be the degree of $f$. Then the degree of $F$ is $d-2$. Since $d$ is not a power of two, we have $d\ge 3$, so that $q\ge 2^7$. Suppose that $F$ has an absolutely irreducible factor over $\FF_q$. Then, by the Hasse-Weil bound (see~\cite[Theorem~5.4.1]{FriJar2008}, for example), the number of $\FF_q$-rational points on the affine curve defined by $F(X,Y)=0$ is at least
\[
q-(d-3)(d-4)q^{1/2}-d+3.
\]
Since $F(1,Y)$ and $F(X,0)$ are polynomials of degree at most $d-2$, the number of $\FF_q$-rational points that are not on the lines $X=1$ or $Y=0$ is at least
\[
q-(d-3)(d-4)q^{1/2}-3d+7,
\]
which (since $d\le q^{1/4}$ and $q\ge 2^7$) is positive. The discussion preceding the proposition then implies that $f$ is not planar.
\end{proof}
\par
The difficulty in applying Proposition~\ref{pro:Hasse_Weil} is to show that the polynomial~$F$ has an absolutely irreducible factor over $\FF_q$. Our strategy will be to apply certain transformations repeatedly to $F$ and then use the following lemma. 
\begin{lemma}
\label{lem:geometric_transform}
Let $G\in\FF_q[X,Y]$ be a nonzero polynomial and define
\[
H(X,Y)=\frac{G(X,XY)}{X^n},
\]
where $n$ is the smallest degree of a monomial in $G$. If $H$ has an absolutely irreducible factor in $\FF_q[X,Y]$, then $G$ has an absolutely irreducible factor in~$\FF_q[X,Y]$.
\end{lemma}
\begin{proof}
Suppose that $H$ has an absolutely irreducible factor over $\FF_q$. We may as well suppose that $H$ itself is absolutely irreducible. Assume that we can factor $G$ as $G=AB$, where $A,B\in\FF_{q^r}[X,Y]$ for some positive integer $r$ and $A$ and $B$ have positive degree. Then we have
\begin{equation}
H(X,Y)=\frac{A(X,XY)}{X^a}\,\frac{B(X,XY)}{X^b}   \label{eqn:factorisation_of_H}
\end{equation}
for some nonnegative integers $a$ and $b$ satisfying $a+b=n$. If $A=\gamma X^a$ or $B=\gamma X^b$ for some nonzero $\gamma\in\FF_{q^r}$, then clearly $G$ has an absolutely irreducible factor in $\FF_q[X,Y]$. Otherwise, both of the factors on the right-hand side of~\eqref{eqn:factorisation_of_H} have positive degree, contradicting that $H$ is absolutely irreducible.
\end{proof}
\par
Now let $H\in\FF_q[X,Y]$ be a polynomial and let $P=(x_0,y_0)$ be a point in the plane. Write
\[
H(X+x_0,Y+y_0)=H_0(X,Y)+H_1(X,Y)+H_2(X,Y)+\cdots,
\]
where $H_i$ is either the zero polynomial or a homogeneous polynomial of degree $i$. If $H_m\ne 0$ and $H_i=0$ for all $i<m$, then the polynomial $H_m$ is called the \emph{tangent cone} of~$F$ at~$P$. Whenever we refer to the tangent cone of a polynomial without specific reference to a point, we mean the tangent cone at the origin~$(0,0)$. 
\par
The following lemma gives a simple criterion for the existence of an absolutely irreducible factor over $\FF_q$ of a polynomial in $\FF_q[X,Y]$. 
\begin{lemma}
\label{lem:tangent}
Let $H\in\FF_q[X,Y]$ and suppose that the tangent cone of $H$ contains a reduced linear factor over $\FF_q$. Then $H$ has an absolutely irreducible factor over $\FF_q$.
\end{lemma}
\par
\begin{proof}
Note that the tangent cone of the product of two polynomials is the product of the individual tangent cones. Therefore, we may assume without loss of generality that $H$ is irreducible in $\FF_q[X,Y]$, otherwise consider the appropriate factor of $H$ in $\FF_q[X,Y]$. By a routine argument (see~\cite{KopYek2008}, for example) there exists $c\in\FF_q$ and an absolutely irreducible polynomial $h\in\FF_{q^r}[X,Y]$ for some positive integer $r$ such that
\[
H=c\prod_{\sigma\in\Gal(\FF_{q^r}/\FF_q)}\sigma(h),
\]
where $\sigma(h)$ means that $\sigma$ is applied to the coefficients of $h$. Letting $T\in\FF_q[X,Y]$ be the tangent cone of $H$ and $t\in\FF_{q^r}[X,Y]$ be the tangent cone of~$h$, we have
\[
T=c\prod_{\sigma\in\Gal(\FF_{q^r}/\FF_q)}\sigma(t).
\]
Since $T$ contains a reduced linear factor over $\FF_q$, there is a unique $\sigma\in\Gal(\FF_{q^r}/\FF_q)$ such that $\sigma(h)$ is divisible by this factor. But since this factor is in $\FF_q[X,Y]$, it divides $\sigma(t)$ for every $\sigma\in\Gal(\FF_{q^r}/\FF_q)$. This forces $r=1$ and thus $H$ is already absolutely irreducible.
\end{proof}
\par
The following proposition combines Proposition~\ref{pro:Hasse_Weil} and Lemmas~\ref{lem:geometric_transform} and \ref{lem:tangent} and summarises the main tool in our proof of Theorem~\ref{thm:planar}. 
\begin{proposition}
\label{pro:main_tool}
Let $f\in\FF_q[X]$ be a polynomial of degree at most $q^{1/4}$ in which the degree of every monomial is not a power of two. Suppose that after the application of a sequence of variable substitutions and transformations of the form $g(X,Y)\mapsto g(X,XY)/X^n$, where $n$ is the smallest degree of a monomial in $g$, to the associated polynomial $F$, we arrive at a polynomial whose tangent cone contains a reduced linear factor over~$\FF_q$. Then~$F$ has an absolutely irreducible factor over~$\FF_q$ and consequently $f$ cannot be planar.
\end{proposition}
\par
We shall also frequently use the following corollary to Lucas's theorem (see~\cite{Fin1947}, for example).
\begin{lemma}
\label{lem:Lucas}
The binomial coefficient $\tbinom{n}{m}$ is even if and only if at least one of the base-$2$ digits of $m$ is greater than the corresponding digit of $n$.
\end{lemma}
\par
A consequence of Lemma~\ref{lem:Lucas} is that, if $i$ is not a power of two, then $X^{2^{\nu(i)}}Y^{d-i}$ is the monomial of smallest degree in $F(X,Y)$ with coefficient~$A_i$, where $\nu(i)$ is the $2$-adic valuation of~$i$. Note also that the only monomial in $F$ of the form $Y^i$ is $Y^{d-2}$. We shall frequently use these facts without specific reference in our proof of Theorem~\ref{thm:planar}.

\section{Proof of Theorem~\ref{thm:planar}}

As before, we assume that $f\in\FF_q[X]$ is a nonzero polynomial in which the degree of every monomial is not a power of two. We now assume in addition that the degree of $f$ is at most~$q^{1/4}$ and that $f$ is planar on $\FF_q$. We show that this leads to a contradiction.
\par
We shall study the associated polynomial $F$ given in~\eqref{eqn:def_F}. Put $F_0=F$ and define $F_1,F_2,\dots,F_t$ recursively by
\[
F_r(X,Y)=\frac{F_{r-1}(XY,Y)}{Y^{n_r}},
\]
where $n_r$ is the smallest degree of a monomial in $F_{r-1}$ and $t$ is the smallest number such that the tangent cone of $F_t$ (at the origin) is not divisible by~$X$. This $t$ exists because of the presence of the monomial $Y^{d-2}$ in $F$. Since $f$ is not a $2$-polynomial, we also have $t\ge 1$. Define $u$ to be the smallest integer such that the tangent cone of $F_{t-1}$ contains the monomial $X^{2^u}Y^\ell$ for some~$\ell$ (by ``contain'' we mean that the monomial is present with some nonzero coefficient).
\par
In the first part of the proof we show that the tangent cone of $F_t$ equals $Y^{2^u-2}$. To do so, we consider the polynomial $G(X,Y)=F(X+1,Y)$, so that
\[
G(X,Y)=Y^{d-2}+\sum_{i=3}^dA_i\frac{X^i+1+(X+1)^i}{X}Y^{d-i}.
\]
For $3\le i\le d$ and $1\le k\le i-1$, the coefficient of $X^{k-1}Y^{d-i}$ in $G$ is $A_i\binom{i}{k}$. Lemma~\ref{lem:Lucas} then implies that, if $i$ is not a power of two, then
\[
X^{2^{\nu(i)}-1}Y^{d-i}
\]
is the monomial of smallest degree in~$G$ with coefficient~$A_i$ (recall that $\nu(i)$ is the $2$-adic valuation of~$i$). Put $G_0=G$ and define $G_1,G_2,\dots,G_t$ recursively by
\[
G_r(X,Y)=\frac{G_{r-1}(XY,Y)}{Y^{n_r-1}},
\]
where $n_1,n_2,\dots,n_r$ are the same numbers that occur in the definition of $F_1,F_2,\dots,F_t$. Note that $n_r-1$ is the smallest degree of a monomial in~$G_{r-1}$, so that we can apply Proposition~\ref{pro:main_tool} to $G_1,G_2,\dots,G_t$.
\par
In order to prove that the tangent cone of~$F_t$ equals $Y^{2^u-2}$, we require the following two lemmas 
\begin{lemma}
\label{lem:u_not_one}
We have $2\le u\le \nu(d)$.
\end{lemma}
\begin{proof}
By assumption, $f$ contains no monomials whose degree is a power of two. In particular $d$ is not a power of two. Hence $F$, and therefore also~$F_{t-1}$, contains the monomial $X^{2^{\nu(d)}}$. Thus we have $u\le \nu(d)$.
\par
If $u=0$, then the tangent cone of $F_{t-1}$ would be divisible by $X$, but not by~$X^2$. This leads to a contradiction by Proposition~\ref{pro:main_tool}.
\par
Now suppose that $u=1$. Then the tangent cone of $F_{t-1}$ contains the monomial $X^2Y^\ell$ for some $\ell$ and does not contain the monomial $XY^{\ell+1}$. By the remarks preceding the lemma, for $r<t$, the tangent cone of $G_r$ equals the tangent cone of~$F_r$ divided by~$X$. Therefore the tangent cone of $G_{t-1}$ is divisible by~$X$ and not by $X^2$. This again leads to a contradiction by Proposition~\ref{pro:main_tool}.
\end{proof}
\par
\begin{lemma}
\label{lem:divisibility_of_nr}
For all $r\le t$ we have $2^u\mid n_r$ .
\end{lemma}
\begin{proof}
By Lemma~\ref{lem:u_not_one} we have $u\le \nu(d)$, and so $2^u\mid d$. Let $s$ be an integer satisfying $0\le s\le t-1$ and assume that $2^u\mid n_r$ for all $r\le s$, which is vacuously true for $s=0$. We proceed by induction on $s$. Recall that $F_{t-1}$ contains the monomial $X^{2^u}Y^\ell$ for some $\ell$. The preimage in~$F_s$ of this monomial is of the form
\begin{equation}
X^{2^u}Y^{2^us-n_1-\cdots-n_s+d-i}   \label{eqn:monomial_Fs_1}
\end{equation}
for some $i$ satisfying $\nu(i)=u$. By the inductive hypothesis, $2^u$ divides the degree of~\eqref{eqn:monomial_Fs_1}. Now suppose that $F_s$ also contains a monomial
\begin{equation}
X^{2^{\nu(j)}}Y^{2^{\nu(j)}s-n_1-\cdots-n_s+d-j}   \label{eqn:monomial_Fs_2}
\end{equation}
of degree smaller than the degree of~\eqref{eqn:monomial_Fs_1}. If $\nu(j)<u$, then by looking at the image in $F_{t-1}$ of~\eqref{eqn:monomial_Fs_2}, we find a contradiction to the minimality of $u$. Otherwise, $2^u$ divides the degree of~\eqref{eqn:monomial_Fs_2} and so $2^u$ divides the smallest degree of a monomial in $F_s$. Hence $2^u\mid n_{s+1}$, as required.
\end{proof}
\par
We now show that the tangent cone of $F_t$ equals $Y^{2^u-2}$.
\begin{lemma}
\label{lem:tangent_cone_Ft}
The tangent cone of $F_t$ is $Y^{2^u-2}$.
\end{lemma}
\begin{proof}
Notice that, since the tangent cone of $F_t$ is not divisible by $X$, it must contain the image of the monomial $Y^{d-2}$ in $F$, namely $Y^{d-2-n_1-\cdots-n_t}$. Since $2^u\mid d$ by Lemma~\ref{lem:u_not_one} and $2^u\mid n_r$ for all $r\le t$ by Lemma~\ref{lem:divisibility_of_nr}, we find that that the tangent cone of $F_t$ contains $Y^j$ for some $j$ satisfying $j\equiv -2\pmod {2^u}$. 
\par
By definition, the tangent cone of $F_{t-1}$ contains $X^{2^u}Y^\ell$ for some $\ell$, and therefore~$F_t$ contains $X^{2^u}$. Hence the tangent cone of $F_t$ has degree at most~$2^u$. Since we also have $u\ge 2$ by Lemma~\ref{lem:u_not_one}, we find that $j=2^u-2$. Hence the tangent cone of $F_t$ has degree $2^u-2$.
\par
Now suppose for a contradiction that the tangent cone of $F_t$ contains $X^{2^v}Y^{2^u-2-2^v}$ for some integer $v$ satisfying $0\le v<u$. By Lemma~\ref{lem:divisibility_of_nr}, the preimage in $F$ of this monomial is of the form $X^{2^v}Y^\ell$ for some $\ell$ satisfying $\ell\equiv -2\pmod {2^v}$. If $v\ge 2$, then Lemma~\ref{lem:Lucas} implies that $F$ also contains $X^2Y^\ell$, whose image in $F_t$ has degree strictly smaller than $2^u-2$, a contradiction. If $v=1$, then the tangent cone of $F_t$ equals
\[
\alpha X^2Y^{2^u-4}+\beta XY^{2^u-3}+Y^{2^u-2}
\]
for some $\alpha,\beta\in\FF_q$ with $\alpha\ne 0$, and the tangent cone of $G_t$ equals
\[
\alpha XY^{2^u-4}+\beta Y^{2^u-3}=Y^{2^u-4}(\alpha X+\beta Y),
\]
which gives a contradiction by Proposition~\ref{pro:main_tool}. If $v=0$, then the tangent cone of $F_t$ must be $Y^{2^u-3}(\beta X+Y)$ for some $\beta\in\FF_q$ with $\beta\ne 0$, which again gives a contradiction by Proposition~\ref{pro:main_tool}.
\end{proof}
\par
In view of Lemma~\ref{lem:tangent_cone_Ft}, define
\[
F_{t+1}(X,Y)=\frac{F_t(X,XY)}{X^{2^u-2}}.
\]
Then $F_{t+1}$ still contains $Y^{2^u-2}$ and the tangent cone of $F_{t+1}$ has degree $2$ and contains $X^2$ (coming from $X^{2^u}Y^\ell$ in $F_{t-1}$). Note that, since $Y^{2^u-2}$ is the unique monomial of degree $2^u-2$ in $F_t$, the only monomial of the form~$Y^i$ in $F_{t+1}$ is $Y^{2^u-2}$.
\par
Now define
\[
F_{t+2}(X,Y)=\frac{F_{t+1}(XY^{2^{u-1}-2},Y)}{Y^{2^u-4}}.
\]
Note that $F_{t+2}$ is obtained from $F_{t+1}$ by $2^{u-1}-2$ applications of the transformation $g(X,Y)\mapsto g(XY,Y)/Y^2$. Also, in each step the smallest degree of a monomial is $2$: a constant term cannot appear because $F_{t+1}$ contains only one monomial that is pure in $Y$, namely $Y^{2^u-2}$, and a linear tangent cone would lead to a contradiction by Proposition~\ref{pro:main_tool}. The tangent cone of~$F_{t+2}$ contains $X^2$ and $Y^2$. We now show that it does not contain~$XY$.
\begin{lemma}
The tangent cone of $F_{t+2}$ equals~$\alpha X^2+Y^2$ for some nonzero $\alpha\in\FF_q$.
\end{lemma}
\begin{proof}
A monomial $X^kY^j$ in $F_t$ is mapped to $X^{k+j-2^u+2}Y^j$ in $F_{t+1}$ and to
\[
X^{k+j-2^u+2}Y^{(2^{u-1}-2)(k+j-2^u+2)+j-2^u+4}
\]
in $F_{t+2}$. Now suppose, for a contradiction, that the latter monomial is $XY$, which means that $k=2^{u-1}$ and $j=2^{u-1}-1$. Since $u\ge 2$ by Lemma~\ref{lem:u_not_one} and $2^u\mid n_r$ for all $r\le t$ by Lemma~\ref{lem:divisibility_of_nr}, the corresponding monomial in $F$ is $X^{2^{u-1}}Y^\ell$ for some odd $\ell$. Lemma~\ref{lem:Lucas} then implies that $F$ also contains the monomial $XY^\ell$ with the same nonzero coefficient as $X^{2^{u-1}}Y^\ell$. However, since $u\ge 2$ and $t\ge 1$, the image in $F_t$ of $XY^\ell$ has degree strictly smaller than $2^u-1$. This contradicts Lemma~\ref{lem:tangent_cone_Ft}, namely that the tangent cone of $F_t$ equals $Y^{2^u-2}$.
\end{proof}
\par
In the remainder of our proof of Theorem~\ref{thm:planar} we shall apply further transformations to $F_{t+2}$, which will ultimately lead to a contradiction. To do so, we first study the images in $F_{t+2}$ of the monomials in $F$. We record the properties of these images in the following lemma. 
\begin{lemma}
\label{lem:images_monomials}
Suppose that $F$ contains the monomial $X^kY^{d-i}$. Then its image in $F_{t+2}$ is $X^rY^s$, where
\begin{align*}
r&=k(t+1)-i+2,\\
s&=k(2^{u-1}(t+1)-t-2)-i(2^{u-1}-1)+2^u.
\end{align*}
\end{lemma}
\begin{proof}
The image in $F_t$ of $X^kY^{d-i}$ is
\[
X^kY^{kt+d-i-n_1-\cdots-n_t}=X^kY^{kt-i+2^u},
\]
since
\[
\sum_{r=1}^tn_r=(d-2)-(2^u-2)=d-2^u,
\]
using Lemma~\ref{lem:tangent_cone_Ft}. Then the image in $F_{t+1}$ of $X^kY^{d-i}$ is
\[
X^{k(t+1)-i+2}Y^{kt-i+2^u}
\]
and the image in $F_{t+2}$ is
\[
X^{k(t+1)-i+2}Y^{kt-i+2^u+(2^{u-1}-2)(k(t+1)-i+2)-2^u+4}.\qedhere
\]
\end{proof}
\par
Lemma~\ref{lem:images_monomials} implies that the putative monomial $X^kY^{d-i}$ in $F$ is mapped to a monomial in $F_{t+2}$ of degree
\[
k(2^{u-1}(t+1)-1)-2^{u-1}i+2^u+2.
\]
In particular, since $u\ge 2$ by Lemma~\ref{lem:u_not_one}, this degree is congruent to $k$ modulo~$2$. Define
\begin{align*}
o(i)&=\begin{cases}
2^{u-1}(t+1)-1-2^{u-1}i+2^u+2 & \text{for odd $i$}\\
(2^{\nu(i)}+1)(2^{u-1}(t+1)-1)-2^{u-1}i+2^u+2 & \text{for even $i$}
\end{cases}
\intertext{and}
e(i)&=\begin{cases}
2^z(2^{u-1}(t+1)-1)-2^{u-1}i+2^u+2 & \text{for odd $i$}\\
2^{\nu(i)}(2^{u-1}(t+1)-1)-2^{u-1}i+2^u+2 & \text{for even $i$},
\end{cases}
\end{align*}
where $z$ is determined as follows. If $i=\sum_{n\ge 0}a_n2^n$ with $a_n\in\{0,1\}$ is the base-$2$ expansion of $i$, then $z$ is the smallest positive integer $n$ such that $a_n=0$. Note that $z\ge 1$ for odd $i$.
\par
Recall our assumption that $f$ contains no monomials whose degree is a power of two and that $f$ is not the zero polynomial. Lemmas~\ref{lem:Lucas} and~\ref{lem:images_monomials} imply that, among all monomials with coefficient $A_i$ in $F_{t+2}$, the smallest odd degree is $o(i)$ and, if $i+1$ is not a power of two, then the smallest even degree is $e(i)$ (if $i+1$ is a power of two, then there are no monomials in $F_{t+2}$ of even degree with coefficient $A_i$). Accordingly, define
\begin{equation}
m=\min\{o(i):A_i\ne 0\}.   \label{eqn:def_m}
\end{equation}
We shall first prove some properties of this number.
\begin{lemma}
\label{lem:unique_monomial}
We have $m=o(i)$ for some uniquely determined $i$.
\end{lemma}
\begin{proof}
Suppose for a contradiction that $m=o(i)=o(i')$ for some integers $i\ne i'$. We first show that one of $i$ and $i'$ is odd and the other is even. If~$i$ and $i'$ are both odd or more generally $\nu(i)=\nu(i')$, then we force $i=i'$, a contradiction. If $i$ and $i'$ are both even and $\nu(i)<\nu(i')$, then we obtain using $u\ge 2$ by Lemma~\ref{lem:u_not_one}
\[
o(i')-o(i)\equiv 2^{\nu(i)}\pmod{2^{\nu(i)+1}},
\]
contradicting $o(i)=o(i')$. This proves our claim and so we can assume without loss of generality that $i$ is even and $i'$ is odd.
\par
Next we show that $e(i)=2$. If there is an even $j$ such that $e(j)<e(i)$ and $A_j\ne 0$, then it follows immediately from the definitions that $o(j)<o(i)$, which contradicts $m=o(i)$. If there is an odd~$j$ such that $e(j)<e(i)$ and $A_j\ne 0$, then
\[
o(j)<e(j)<e(i)<o(i),
\]
which again contradicts $m=o(i)$. Hence $e(i)$ is the smallest even degree of a monomial in $F_{t+2}$. Since $F_{t+2}$ contains the monomial $X^2$ and no monomials of smaller degree, we find that $e(i)=2$.
\par
Since $i$ is even and $e(i)=2$, we obtain $o(i)=2^{u-1}(t+1)+1$. The equality $o(i)=o(i')$ then gives $2^{u-1}i'=2^u$, so $i'=2$, contradicting that $i'$ is odd.
\end{proof}
\par
\begin{lemma}
\label{lem:even_degrees}
Every monomial in $F_{t+2}$ of degree strictly less than $m$ has even degree in $X$ and even degree in $Y$.
\end{lemma}
\begin{proof}
Let $i$ be an integer such that $A_i\ne 0$. If $i$ is odd, then $o(i)<e(i)$, and so $o(i)$ is the smallest degree of a monomial in $F_{t+2}$ with coefficient $A_i$. Hence, if there is a monomial in $F_{t+2}$ of even degree less than $o(i)$ with the same coefficient~$A_i$, then $i$ must be even. By Lemma~\ref{lem:images_monomials}, such a monomial has degree
\[
k(2^{u-1}(t+1)-1)-2^{u-1}i+2^u+2.
\]
Since $u\ge 2$ by Lemma~\ref{lem:u_not_one}, we force $k$ to be even. It then follows from Lemma~\ref{lem:images_monomials} that this monomial has even degree in $X$ and even degree in~$Y$. 
\end{proof}
\par
We now complete the proof of Theorem~\ref{thm:planar}.
\begin{proof}[Proof of Theorem~\ref{thm:planar}]
Recall the definition of $m$ from~\eqref{eqn:def_m}. Put $H_0=F_{t+2}$ and define $H_1,H_2,\dots,H_{(m-1)/2}$ recursively by
\[
H_{i+1}(X,Y)=\frac{H_i(X,c_iX+XY)}{X^2},
\]
where $c_i$ is such that $c_i^2$ is the coefficient of $X^2$ in $H_i$. Note that $H_{i+1}$ is obtained from $H_i$ by a variable substitution $(X,Y)\mapsto (X,c_iX+Y)$ followed by the transformation $g(X,Y)\mapsto g(X,XY)/X^2$. We shall see that $H_1,H_2,\dots,H_{(m-1)/2}$ are indeed polynomials and that the tangent cone of $H_{(m-1)/2}$ equals $\alpha X$ for some nonzero $\alpha\in\FF_q$, which then leads to a contradiction by Proposition~\ref{pro:main_tool}.
\par
By Lemma~\ref{lem:unique_monomial}, the polynomial $H_0=F_{t+2}$ contains a unique monomial of degree $m$, and so $H_0(X,c_0X+Y)$ contains $\alpha X^m$ for some nonzero $\alpha\in\FF_q$. Lemma~\ref{lem:even_degrees} asserts that every monomial in~$H_0$ of degree strictly less than~$m$ has even degree in $X$ and even degree in $Y$. Since $\binom{n}{k}$ is even for even~$n$ and odd $k$ by Lemma~\ref{lem:Lucas}, the images of such monomials in $H_1,H_2,\dots,H_{(m-1)/2-1}$ also have even degree in $X$ and even degree in $Y$. Hence the tangent cones of $H_1,H_2,\dots,H_{(m-1)/2-1}$ have degree two and never contain $XY$. This also implies that $H_1,H_2,\dots,H_{(m-1)/2}$ are indeed polynomials and that the tangent cone of $H_{(m-1)/2}$ equals $\alpha X$.
\end{proof}


\section{Proof of Proposition~\ref{pro:APN}}
\label{sec:APN}

We now give a proof of Proposition~\ref{pro:APN}. As before, let $q$ be a power of two and let $f\in\FF_q[X]$ be a polynomial in which the degree of every monomial is not a power of two. Again there is no loss of generality since the addition of a 2-polynomial preserves the APN property of the function induced by $f$. Define the polynomial
\[
\psi(X,Y,Z)=\frac{f(X)+f(Y)+f(Z)+f(X+Y+Z)}{(X+Y)(X+Z)(Y+Z)}.
\]
It is well known (see~\cite[Proposition 3.1]{Rod2009}, for example) that $f$ induces an APN function on $\FF_q$ if and only if all $\FF_q$-rational points on the affine surface defined by $\psi(X,Y,Z)=0$ satisfy $(X+Y)(X+Z)(Y+Z)=0$. Write
\[
f=\sum_{i=0}^dA_iX^i,
\]
where $A_d\ne 0$. Then the homogenised form of $\psi$ is
\[
\widetilde{\psi}(X,Y,Z,T)=\sum_{i=3}^dA_i\frac{X^i+Y^i+Z^i+(X+Y+Z)^i}{(X+Y)(X+Z)(Y+Z)}T^{d-i}.
\]
As for planar functions, we consider the affine curve defined by $F(X,Y)=0$, where $F(X,Y)=\widetilde{\psi}(X,1,X+1,Y)$. We have
\[
F(X,Y)=\sum_{i=3}^dA_i\frac{X^i+1+(X+1)^i}{(X+1)X}Y^{d-i}
\]
and, after expanding,
\[
F(X,Y)=\sum_{i=3}^dA_iY^{d-i}\sum_{k=1}^{i-1}\Bigg[\binom{i-1}{k}+1\Bigg]X^{k-1}. \]
If $f$ induces an APN function on $\FF_q$, then all $\FF_q$-rational points of the affine curve defined by $F(X,Y)=0$ satisfy $XY(X+1)=0$.
\par
Now assume that $d\le q^{1/4}$ and $d\equiv 2\pmod 4$. Then $F(0,0)=0$ and Lemma~\ref{lem:Lucas} implies that the tangent cone of $F$ equals $A_dX+A_{d-1}Y$. Since $A_d\ne 0$, we find from Lemma~\ref{lem:tangent} that $F$ has an absolutely irreducible factor over $\FF_q$. An argument that is almost identical to that used in the proof of Proposition~\ref{pro:Hasse_Weil} then shows that the curve defined by $F(X,Y)=0$ has $\FF_q$-rational points not on one of the lines $X=0$, $Y=0$, or $X=1$. Hence $f$ cannot be APN on $\FF_q$.

\section*{Acknowledgements}

\sloppypar

Daniele Bartoli was partially supported by the Italian Ministero dell'Istruzione, dell'Universit\`a e della Ricerca (MIUR) and the Gruppo Nazionale per le Strutture Algebriche, Geometriche e le loro Applicazioni (GNSAGA-INdAM). This work was carried out when the first author was visiting Paderborn University under the programme ``Research Stays for University Academics and Scientists" funded by the German Academic Exchange Service (DAAD).


\providecommand{\bysame}{\leavevmode\hbox to3em{\hrulefill}\thinspace}
\providecommand{\MR}{\relax\ifhmode\unskip\space\fi MR }
\providecommand{\MRhref}[2]{%
  \href{http://www.ams.org/mathscinet-getitem?mr=#1}{#2}
}
\providecommand{\href}[2]{#2}

\end{document}